\documentclass[12pt]{amsart}
\usepackage{graphicx} 

\usepackage{BAstyle}
\usepackage[margin=1in]{geometry}
\DeclareMathOperator{\cond}{cond}

\DeclareMathOperator{\PGL}{PGL}
\DeclareMathOperator{\PSL}{PSL}
\newcommand{\PP}{\mathbb{P}}

\title[Inertial Multiplicity Bounds]{Inertial multiplicity bounds for two dimensional projective representations and bounds for the number of $\operatorname{PSL}_2(\mathbb{F}_q)$ and $\operatorname{PGL}_2(\mathbb{F}_q)$ number fields}
\author{Brandon Alberts}

\begin{document}

\begin{abstract}
    We prove upper bounds for certain number field counting functions using Serre's modularity conjecture (now a theorem of Khare--Wintenberger). These results are comparable to sharp upper bounds for the number of abelian extensions with fixed or bounded discriminant, with Serre's modularity conjecture playing the role of class field theory.
\end{abstract}

\maketitle

\section{Introduction}

We give a bound for the number of odd irreducible two dimensional projective representations over $\F_q$ with prescribed tame inertia. Let $G_\Q$ denote the absolute Galois group over $\Q$ and $I_p$ denote the inertia group at a prime $p$, which is well-defined up to conjugacy as a subgroup of $G_\Q$. If $G$ is a finite group, $\rho:G_\Q\to G$ is a continuous surjective homomorphism, and $C\le G$ is a cyclic subgroup, let $\inv_{G,C}(\rho)$ denote the product of primes $p\nmid |G|$ for which $\rho(I_p)$ is conjugate to $C$.

The following is a pointwise inertial multiplicity bound for irreducible subgroups of $\PGL_2(\F_q)$, up to a ramification condition at $\infty$ when $q$ is odd.
\begin{theorem}\label{thm:main}
    Let $q=\ell^a$ be a prime power, $V=\F_q^2$, and $G\le \PGL_2(\F_q) = \PGL(V)$ an irreducible subgroup. Then for each tuple $(N_C)$ indexed by conjugacy classes of nontrivial cyclic subgroups $C\le G$,
    \[
        \#\left\{\bar{\rho}:G_\Q\to \PGL_2(\F_q) : \substack{\displaystyle\bar\rho \text{ is odd, irreducible, }\bar\rho(G_\Q)=G\text{, and}\\\displaystyle\text{for each cyclic }C\le G,\ \inv_{G,C}(\bar{\rho}) = N_C}\right\} \ll_{q} \prod_C N_C^{2e(G,C)},
    \]
    where
    \[
        e(G,C) = \begin{cases}
            2 & C\text{ is semisimple (that is its action on }\PP^1(\overline\F_\ell)\text{ has two fixed points)}\\
             &\hspace{0.5cm} \text{ and }G\text{ contains an element swapping the two fixed points of }C\\
            1 & \text{else.}
        \end{cases}
    \]
\end{theorem}
The proof is short and not particularly novel. A lifting theorem of Tate implies that any such $\bar{\rho}$ can be lifted to an odd irreducible representation $\rho:G_\Q\to \GL_2(\F_{q^d})$ for some $d\ge 1$ with prescribed ramification. The result then follows directly from Serre's modularity conjecture together with bounds for the dimension of the space of newforms. An average form of this result for non-projective representations over prime power levels appears in \cite{centeleghe2011computing}, which uses Serre's modularity conjecture in essentially the same way. Yet, when phrased as a multiplicity bound Theorem \ref{thm:main} has striking applications to number field counting. The purpose of this note is to be a proof of concept for applications of Serre's modularity conjecture (and therefore the Langlands program in general) to number field counting.

For any finite transitive group $G$, define the counting function
\[
    \mathcal{F}_{\disc,\Q}(G;X) = \{\rho:G_\Q\to G:\text{surjective, }\disc(\overline{\Q}^{\rho^{-1}(\Stab_G(1))}/\Q)\le X\}.
\]
Given a projective faithful irreducible representation $G\hookrightarrow \PGL_2(\F_q)$, we define $\mathcal{F}_{\disc,\Q}^{\rm odd}(G;X)$ to be the subset of such homomorphisms which are odd, that is $\det\rho(c) = -1$ where $c\in G_\Q$ is complex conjugation. The ``odd'' condition is really a local restriction at the infinite place.

The following is an immediate consequence of Theorem \ref{thm:main}.
\begin{corollary}\label{cor:PSL2_PGL2}
    Let $q$ be a prime power. Let $G$ be any irreducible subgroup of $\PGL_2(\F_q)$ in any one of its degree $n$ faithful transitive representations, such as $\PSL_2(\F_q)$ or $\PGL_2(\F_q)$ itself. Let $\beta(G) = \max_{g\ne 1} \frac{2e(G,\langle g\rangle) + 1}{\ind_n(g)}$, where $\ind_n(g) = n - \#\{\text{orbits of }g\}$.
    \begin{enumerate}[(i)]
        \item If $2\mid q$, then $\#\mathcal{F}_{\disc,\Q}(G;X) \ll_{q,\epsilon} X^{\beta(G)+\epsilon}$.
        \item If $2\nmid q$, then $\#\mathcal{F}^{\rm odd}_{\disc,\Q}(G;X) \ll_{q,\epsilon} X^{\beta(G)+\epsilon}$.
    \end{enumerate}
    In particular, $3/a(G) \le \beta(G) \le 5/a(G)$, where $a(G) = \min_{g\ne 1}\ind_n(g)$.
\end{corollary}

The state of the art for uniform exponent bounds in general is work of Lemke Oliver \cite{lemkeoliver2023uniform}. Lemke Oliver leveraged the finer structure of finite primitive permutation groups to prove several uniform exponent bounds for the number of extensions with discriminant $\le X$, and proved several results balancing various methods of doing so. For linear groups of rank $m$, in particular the simple linear groups $\PSL_2(\F_q)$ of rank $2$, Lemke Oliver proves bounds of the form $\ll_q X^{O_m(1)}$. The bounds we prove in Corollary \ref{cor:PSL2_PGL2} are $\ll_q X^{O(1/q)}$, which follows from $a(G) \gg q$ for each of these groups (see \cite[Example on page 16]{lemkeoliver2023uniform}). These bounds are within a fixed power of Malle's predicted bound, $X^{1/a(G)+\epsilon}$ \cite{malle2002distribution}. In particular, $\PSL_2(\F_{2^d})$ with $d > 1$ is an infinite family of nonabelian simple groups for which Corollary \ref{cor:PSL2_PGL2} proves an upper bound, without a local restriction at $\infty$, that is within a fixed power of Malle's predicted bound.

Instead of comparing with Lemke Oliver's results, which are the best known for general permutation groups, our results more closely resemble those of abelian groups: when $G$ is abelian, Malle's predicted upper bound is proven by M\"aki over $\Q$ \cite{maki1985density}, and by Wright in general \cite{wright1989distribution}. The proofs are via class field theory, where $G$ is viewed as a product of subgroups of $\GL_1(\F_q)$ (i.e., a product of cyclic groups). The upper bounds stemming from their method can be interpreted as a multiplicity bound of the form
\[
    \#\{\rho:G_k\to \GL_1(\F_q) : \rho(G_k)=G\text{, and for each cyclic }C,\ \inv_{G,C}(\rho) = N_C\} \ll_{q,\epsilon} \prod_{C\le G} N_C^{\epsilon}
\]
for an arbitrary number field $k$, where $G\le \GL_1(\F_q)$. The use of Serre's modularity conjecture replaces class field theory with a part of Langlands program in dimension $2$. The Langlands program is colloquially understood as playing the role of ``nonabelian class field theory'', and the main results of this paper provide the first example known to the author of results in the Langlands program in dimension $>1$ being used for number field counting in a similar way to class field theory. (See \cite{lipnowski2016bhargava} for work of Lipnowski in the opposite direction, assuming number field counting results to prove statements about modular representations and elliptic curves).

\begin{remark}
    These results may be combined with existing methods in the literature to inductively produce similar quality bounds for several related groups, including $\GL_2(\F_q)$, its irreducible subgroups, and direct products of two or more of these groups. We refer to the existing literature (such as \cite{ALOWW,lemkeoliver2023uniform}) for the precise methods used to produce upper bounds for such groups from Theorem \ref{thm:main}.
\end{remark}

\section{Proof of the results}
 
Throughout, $q = \ell^a$ is a prime power, $V = \F_q^2$, and $\pi : \GL_2 \to \PGL_2$ is the quotient map. We fix an embedding $\F_q\hookrightarrow \overline\F_\ell$ and view linear representations as valued in $\GL_2(\overline\F_\ell)$. A subgroup $G \le \PGL_2(\F_q)$ is \emph{irreducible} if $G$ has no fixed points in $\PP^1(\overline\F_\ell)$ (equivalently, if no conjugate of $\pi^{-1}(G)$ lies in the upper triangular matrices of $\GL_2(\overline\F_\ell)$). Every lift of an irreducible projective representation is therefore an irreducible linear representation.

We call a projective representation $\bar\rho:G_\Q\to \PGL_2(\overline\F_\ell)$ \emph{odd} if there exists a lift $\rho:G_\Q\to \GL_2(\overline\F_\ell)$ for which $\det\rho(c) = -1$, where $c\in G_\Q$ is complex conjugation. This is well-defined, as every lift is a twist $\rho\otimes \chi$ by some character $\chi:G_\Q\to Z(\GL_2(\overline\F_\ell))\cong \overline\F_\ell^{\times}$, so that
\[
    \det(\rho\otimes\chi)(c) = (\det\rho(c))(\chi(c))^2 = (\det\rho(c))\chi(c^2) = \det\rho(c).
\]
Every projective representation has a lift to $\GL_2$ as we will see below, so oddness is a well-defined property of all projective representations.

We let $e(G,1) = 0$ for convenience, to reflect that $C=1$ and primes with $\bar\rho(I_p)=1$ do not appear in any of the bounds that follow.
 
\subsection{The Artin conductor and tame inertia}
 
For a continuous representation $\rho : G_\Q \to \GL(W)$ on a finite dimensional vector space $W$ over a finite field, the Artin conductor is defined multiplicatively as $\cond(\rho) = \prod_p p^{\nu_p(\cond(\rho))}$, where
\[
    \nu_p(\cond(\rho)) = \sum_{j=0}^{\infty} \frac{|\rho(I_{p,j})|}{|\rho(I_{p,0})|} \dim\left(W/W^{\rho(I_{p,j})}\right)
\]
with $I_{p,j}$ the filtration of higher ramification subgroups in the lower numbering and $I_{p,0} = I_p$ the inertia group at $p$. When $\rho$ is tamely ramified at $p$, that is $\rho(I_{p,j}) = 1$ for $j \ge 1$, only the first term survives and $\nu_p(\cond(\rho)) = \dim(W/W^{\rho(I_p)})$.
 
\subsection{Lifting projective representations}
 
A lifting theorem of Tate shows that all projective representations lift to linear representations, and by twisting one can prescribe the ramification at all finite places. We summarize the consequence we will use in the following lemma.
 
\begin{lemma}\label{lem:tate}
Let $\bar\rho : G_\Q \to \PGL_2(\F_q)$ be continuous, irreducible, and odd, with image $G$. Then there exist $d \ge 1$ and a continuous lift $\rho : G_\Q \to \GL_2(\overline\F_{\ell})$ of $\bar\rho$ that is irreducible, odd, unramified at all primes unramified in $\bar\rho$, and satisfies the following bounds on its Artin conductor:
\[
    \cond(\rho) \ll_{\ell,|G|} \prod_{C\le G} \inv_{G,C}(\bar\rho)^{e(G,C)},
\]
where $e(G,C)$ is defined as in Theorem \ref{thm:main}.
\end{lemma}
 
\begin{proof}
The quotient $\pi:\GL_2(\overline\F_\ell)\to \PGL_2(\overline\F_\ell)$ has central kernel $\ker\pi = Z(\GL_2(\overline\F_\ell)) \cong \overline\F_\ell^{\times}$. Therefore, the obstruction to lifting $\bar\rho$ to $\GL_2(\overline\F_\ell)$ lies in $H^2(G_\Q, \overline\F_\ell^\times)$, which vanishes by a theorem of Tate (see \cite[\S 6.5]{serre1977modular} for a proof that $H^2(G_K,\Q_p/\Z_p) = 0$ for all primes $p$). Thus, a continuous lift $G_\Q\to \GL_2(\overline\F_\ell)$ exists. Moreover, any such lift is odd and irreducible because $\bar\rho$ is odd and irreducible.

The full set of lifts is given by the twist of a single lift by the characters $\chi:G_\Q\to \overline\F_\ell^{\times}$, which by the usual twisting argument can be used to control ramification. Specifically, given any collection $\{\rho_p:G_{\Q_p}\to \GL_2(\overline\F_\ell)\}$ at the finite places which are local lifts of $\bar{\rho}|_{G_{\Q_p}}$ and for which $\rho_p$ is unramified for all but finitely many $p$, there exists a global lift $\rho:G_\Q\to \GL_2(\overline\F_\ell)$ of $\bar\rho$ for which $\rho|_{I_p} = \rho_p|_{I_p}$ for each finite place $p$. See \cite[Theorem 7.2.2]{bosman2011computations} for an explicit reference. As any lift is necessarily odd, we conclude that we may choose a lift with inertia given by any prescribed local lifts (as long as at most finitely many places are ramified).

If $p\nmid \ell|G|\infty$ and $\bar\rho(I_p) = 1$, then there exists a local lift $\rho_p$ of $\bar\rho|_{G_{\Q_p}}$ which is unramified. If $p\nmid \ell|G|\infty$ is ramified in $\bar\rho$ with $\bar\rho(I_p) = C$, there are three cases:
\begin{enumerate}
    \item If $e(G,C) = 2$, that is $C$ is semisimple and $G$ permutes the fixed points of $C$, then any choice of local lift satisfies $\dim(V/V^{\rho_p(I_p)}) \le 2 = e(G,C)$. We claim that there exists a tame local lift, so that $e(G,C)$ is an upper bound for $\nu_p(\cond(\rho_p))$. Given a local lift $\rho_p$, the fact that $\overline{\F}_{\ell}^{\times}$ is abelian together with Kronecker--Weber implies there exists a character $\chi:G_\Q\to \overline{\F}_{\ell}^{\times}$ unramified away from $\{p,\infty\}$ for which $\chi|_{I_p} = \chi_{I_{p,1}} = \rho_p|_{I_{p,1}}$. Thus, twisting by $\chi^{-1}$ gives the tame local lift $\rho_p\otimes \chi^{-1}$.
    
    \item If $e(G,C) = 1$ and $C$ is semisimple, then no element of $G$ can swap the fixed points of $C$. Thus the normalizer of $C$ is the maximal torus containing $C$, which is the same as its centralizer. More concretely, choose a basis of $V\otimes \overline\F_\ell$ which generates the two lines in $\PP^1(\overline\F_\ell)$ fixed by $C$. The torus $T = \langle \left(\begin{smallmatrix} a & 0 \\ 0 & d\end{smallmatrix}\right)\rangle$ is the subgroup of $\GL_2(\overline\F_\ell)$ fixing the two lines, so that the normalizer and centralizer of $C$ are both given by $N_G(C) = C_G(C) = G \cap (T/Z)$, where $Z = \langle aI\rangle$ is the center. Any lift of Frobenius in $T$ necessarily commutes with $C$ by landing in the centralizer, so that the tame relation implies $x^p = x$ for each $x\in C$. Noting that $T/Z = Z\langle (\begin{smallmatrix} 1 & 0 \\ 0 & d\end{smallmatrix})\rangle/Z$, we may choose a lift $\rho_p$ for which $\rho_p(I_p) = \langle(\begin{smallmatrix} 1 & 0\\ 0 & d\end{smallmatrix})\rangle$ with $d$ of order $|C|$, so that $x^p=x$ for any element in this subgroup thus respecting the tame relation with any choice of lift for Frobenius. In this case, $\dim(V/V^{\rho_p(I_p)}) \le 1 = e(G,C)$.

    \item If $e(G,C) = 1$ and $C$ is unipotent, then $|C| = \ell$, which we recall is distinct from $p$. Any element normalizing $C$ fixes the unique fixed point of $C$ on $\PP^1(\overline\F_\ell)$, so $N_G(C)$ is contained in the Borel subgroup $B/Z$ stabilizing this point. In particular $\bar\rho(\mathrm{Frob}_p) \in B/Z$, as it normalizes $\bar\rho(I_p) = C$. Choose a basis in which this fixed point is spanned by $e_1$, so that $C$ is upper triangular unipotent, and for which $h=\left(\begin{smallmatrix}1 & 1\\ 0 & 1\end{smallmatrix}\right)$ is the unique lift of a chosen generator of $C$ having $1$ as an eigenvalue. For any upper triangular lift $F = \left(\begin{smallmatrix}a & \ast\\ 0 & d\end{smallmatrix}\right)$ of $\bar\rho(\mathrm{Frob}_p)$ one computes $F h F^{-1} = h^{a/d}$ which is independent of the choice of lift. As $\langle h\rangle = C$ has order $\ell$, it must be that $a/d \in \F_\ell^\times$. Moreover, $a/d = p$ as $\bar\rho(\mathrm{Frob}_p)$ induces $x \mapsto x^p$ on $C$. Hence $FhF^{-1} = h^p$, and $\sigma \mapsto h$, $\varphi \mapsto F$ defines a tame local lift $\rho_p$ of $\bar\rho|_{G_{\Q_p}}$ with $V^{\rho_p(I_p)} = \langle e_1 \rangle$, so that $\dim(V/V^{\rho_p(I_p)}) = 1 = e(G,C)$.
\end{enumerate}

Thus, if $p\nmid \ell|G|\infty$ then $\nu_p(\cond(\rho)) \le e(G,\bar\rho(I_p))$. Finally, if $p\mid \ell|G|$ choose a lift which minimizes $\nu_p(\cond(\rho))$. We note that there are only finitely many homomorphisms $G_{\Q_p}\to G$, so there are only finitely many possible lifts $\rho_p$ for each $p\mid \ell|G|$. This implies $\nu_p(\cond(\rho)) \ll_{\ell,G} 1$. Put together, we conclude that
\[
    \cond(\rho) \le \prod_{p\mid \ell|G|} p^{O_{\ell,G}(1)} \prod_{p\nmid \ell|G|}p^{e(G,\bar\rho(I_p))} \ll_{\ell,G} \prod_{C\le G} \inv_{G,C}(\bar\rho)^{e(G,C)}.
\]
\end{proof}
 
\subsection{The multiplicity bound}
 
\begin{proof}[Proof of Theorem \ref{thm:main}]
Let $\bar\rho : G_\Q \to \PGL_2(\F_q)$ be odd and irreducible with $\bar\rho(G_\Q) = G$ and $\inv_{G,C}(\bar\rho) = N_C$ for each nontrivial cyclic $C \le G$ up to conjugacy. Choose a lift $\rho : G_\Q \to \GL_2(\overline\F_\ell)$ as in Lemma \ref{lem:tate}. Serre's modularity conjecture, now a theorem of Khare--Wintenberger \cite{khare2009serreI, khare2009serreII}, states that $\rho$ is modular, more specifically that there is a normalized newform $f$ of specific weight $k(\rho)$, level $N(\rho)$, and Nebentypus $\varepsilon(\rho)$ depending on $\rho$ for which $\bar\rho_f \cong \rho$.
 
Consider the map $\bar\rho \mapsto f$ given by choosing a lift $\rho$ as in Lemma \ref{lem:tate} and choosing a newform $f$ of optimal weight and level. This map is nearly injective, with fibers given by $O_G(1)$ many conjugates. Thus, it suffices to bound the number of newforms arising in this way. Serre's recipe gives a weight satisfying $2\le k(\rho)\le \ell^2$, which is bounded in terms of $\ell$. Meanwhile, the level $N(\rho)$ is given by the prime-to-$\ell$ part of the Artin conductor. By Lemma \ref{lem:tate}, we certainly have
\[
    N(\rho) \mid N := \prod_{p\mid\ell|G|}p^{O_{\ell,G}(1)}\prod_{C\le G} N_C^{e(G,C)}.
\]
Thus, it suffices to bound the number of newforms in $\bigcup_{k=2}^{\ell^2} S_k(\Gamma_1(N))$, up to an implied constant depending only on $\ell$ and $G$.

By strong multiplicity $1$ (see \cite[Theorem 5.8.2]{diamond2005first}) distinct newforms have distinct Hecke eigensystems and therefore are linearly independent in the space of cusp forms. In particular, the number of newforms of fixed weight $k$ and level dividing $N$ is bounded above by the dimension $\dim S_k(\Gamma_1(N))$. It then follows from \cite[Theorem 7]{martin2005dimensions} that the number of such newforms is bounded above by
\begin{equation}\label{eq:thm_main_bound}
    \sum_{k=2}^{\ell^2} \dim S_k(\Gamma_1(N))
    \le \sum_{k=2}^{\ell^2} \frac{k-1}{24}N^2 + O(kN) \ll_{\ell} N^2 \ll_{\ell,G} \ \prod_{C\le G} N_C^{2e(G,C)}.
\end{equation}
\end{proof}
 
\subsection{The discriminant bound}
 
\begin{proof}[Proof of Corollary \ref{cor:PSL2_PGL2}]
Let $G \le \PGL_2(\F_q)$ be irreducible and fix a faithful transitive representation of degree $n$ with point stabilizer $\Stab_G(1)$. Up to the wild places, the discriminant of a surjective projective representation $\bar\rho:G_\Q\to G$ is given by
\[
    \disc\left(\overline\Q^{\bar\rho^{-1}(\Stab_G(1))}/\Q\right) \ge \prod_{C\le G} \inv_{G,C}(\bar\rho)^{\ind_n(C)}.
\]
We will apply Theorem \ref{thm:main}, as $G$ is irreducible by definition. We note that all representations are odd in characteristic $2$, as $\det\rho(c) = \pm 1 \equiv 1\pmod 2$, so when $2\mid q$ we have $\mathcal{F}_{\disc,\Q}(G;X) = \mathcal{F}_{\disc,\Q}^{\mathrm{odd}}(G;X)$.

Writing $N_C = \inv_{G,C}(\bar\rho)$ and summing over all $\bar\rho$ giving the same tuple $(N_C)$ as in Theorem \ref{thm:main} implies
\[
    \#\mathcal F^{\mathrm{odd}}_{\disc,\Q}(G;X) \ll_{\ell,G} \sum_{\prod_C N_C^{\ind_n(C)} \le X}\ \prod_C N_C^{2e(G,C)}.
\]
The generating Dirichlet series for the righthand side is multiplicative, given by
\[
    \sum_{(N_C)} \prod_{C\le G} N_C^{2e(G,C) - s\ind_n(C)} = \prod_p \left(1 + \sum_{C\le G} p^{2e(G,C) - s\ind_n(C)}\right).
\]
This converges absolutely whenever ${\rm Re}(s) > \frac{2e(G,C)+1}{\ind_n(C)}$ for each nontrivial cyclic $C\le G$, i.e. ${\rm Re}(s) > \beta(G)$. The result then follows bounding from the integral in Perron's formula on the line ${\rm Re}(s) = \beta(G) + \epsilon$.
 
The bound $3/a(G) \le \beta(G) \le 5/a(G)$ is immediate from $1\le e(G,C) \le 2$ and $\ind_n(C) \ge a(G)$ (where the lower bound comes from choosing $C$ with $\ind_n(C) = a(G)$).
\end{proof}

\section*{Acknowledgements}
The author thanks the American Institute of Mathematics (AIM) for hosting the author at the workshop ``AI and number theory'' as well as Anthropic and Ralph Furman for providing a free trial of the Claude Max subscription to the workshop participants. The author also thanks Robert Lemke Oliver for helpful discussions and feedback on an earlier draft.
 
This paper was prepared with the assistance of AI, specifically Claude Opus 4.8 and Fable 5. The mathematical ideas are entirely original work of the author. AI was used to assist with writing, literature search, and proofreading, greatly streamlining the process of preparing this paper.

\bibliographystyle{abbrv}
\bibliography{main.bbl}
\end{document}